\documentclass[12pt]{amsart}
\pdfoutput=1 
\allowdisplaybreaks 

\usepackage{hyperref}
\usepackage{amsmath,amsfonts,amssymb}
\usepackage{url} 
\makeatletter
\def\url@leostyle{%
   \@ifundefined{selectfont}{\def\UrlFont{\sf}}{\def\UrlFont{\small\ttfamily}}}
\makeatother
\newtheorem{thm}{Theorem}[section]
\newtheorem{cor}[thm]{Corollary}
\newtheorem{prop}[thm]{Proposition}
\newtheorem{lem}[thm]{Lemma}

\newcommand{\Dc}{\mathcal D}

\newcommand{\Fc}{\mathcal F}
\newcommand{\Gc}{\mathcal G}
\newcommand{\Hc}{\mathcal H}

\newcommand{\Lc}{\mathcal L}
\newcommand{\Nc}{\mathcal N}

\newcommand{\Pc}{\mathcal P}
\newcommand{\Rc}{\mathcal R}
\newcommand{\Sc}{\mathcal S}

\newcommand{\Vc}{\mathcal V}
\newcommand{\Zc}{\mathcal Z}

\newcommand{\Db}{\mathbb{D}}

\newcommand{\Nb}{\mathbb{N}}

\newcommand{\Rb}{\mathbb{R}}

\newcommand{\Zb}{\mathbb{Z}}

\newcommand{\nn}[1]{| #1 |}
\newcommand{\nd}[1]{\left\|#1 \right\|}
\newcommand{\tr}{\mbox{\rm tr}  }

\newcommand{\id}{\mbox{\rm Id} }
\newcommand{\radon}{\Rc }
\newcommand{\restrict}{\mbox{\rm Rest}\, }
\newcommand{\radial}{\mbox{\rm Radial}\, }
\newcommand{\im}{\mbox{\rm im}\, }
\newcommand{\nf}{N_{3,2} }
\newcommand{\ncf}{\Nc_{3,2}}

\date{\today}

\title[Gelfand transform of ${\Sc(N_{3,2})}^{SO(3)}$]
{Gelfand transforms\\
of $SO(3)$-invariant Schwartz functions\\
on the free nilpotent group~$N_{3,2}$}

\author[V.~Fischer and F.~ Ricci]
{V\'eronique Fischer and Fulvio Ricci}

\address
{Scuola Normale Superiore di Pisa,\\
Piazza dei Cavalieri, 7\\ I-56126 Pisa}\email{v.fischer@sns.it, f.ricci@sns.it}

\begin{document}

\maketitle

\begin{abstract}
The spectrum of a Gelfand pair $(K\ltimes N, K)$,
where $N$ is a nilpotent group,
can be embedded in a Euclidean space.
We prove that in general,
the Schwartz functions on the spectrum 
are 
the Gelfand transforms of Schwartz
$K$-invariant functions on $N$.
We also show the converse in the case of the Gelfand pair
$(SO(3)\ltimes  \nf, SO(3))$,
where $\nf$ is the free two-step
nilpotent Lie group with three generators.
This extends recent results for the Heisenberg group.
\end{abstract}

\section{Introduction}

Let $N$ be a connected, simply-connected, two-step nilpotent Lie group.
Let $K$ be a compact group acting by automorphism on $N$.
We assume that $(K\ltimes N,K)$ is a Gelfand pair.
The Gelfand spectrum can be homeomorphically embedded
in a Euclidean space as follows.

Let ${\Db(N)}^K$ be the algebra 
of left-invariant and $K$-invariant differential operators on $N$ 
and $\{D_1,\ldots, D_q\}$ a finite set of essentially self-adjoint generators of ${\Db(N)}^K$. 
We call $\Dc$ the ordered family $(D_1,\ldots,D_q)$.
To each bounded $K$-spherical function $\phi$ on $N$
we assign the $q$-tuples of eigenvalues 
$\mu(\phi)=(\mu_1(\phi),\ldots, \mu_q(\phi))$,
i.e. such that $D_j\phi=\mu_j(\phi)\phi$.
The set $\Sigma_\Dc$ of such $q$-tuples is in 1-1 correspondence with
the Gelfand spectrum and the topology induced on it from $\Rb^q$
coincides with the Gelfand topology 
\cite{ferrari}.

We define the Gelfand transform
$\Gc : {L^1(N)}^K \rightarrow  C_o(\Sigma_\Dc)$
by:
$$
\Gc F(\mu(\phi))=\int_N F \bar \phi
.
$$
We are interested in the following conjecture:
\begin{center}
  \textit{$\Gc$ establishes an isomorphism between 
${\Sc(N)}^K$
and $\Sc(\Sigma_\Dc)$\\
(as Fr\'echet spaces)}
\end{center}
The validity of this statement is independent of the choice of~$\Dc$
(see Section~\ref{sec_results});
therefore once proved for one
particular choice of~$\Dc$, it is true for any choice of~$\Dc$.

Proposition~\ref{prop_gen_hula} of this paper shows 
the continuous inclusion $\Sc(\Sigma_\Dc) \hookrightarrow \Gc ({\Sc(N)}^K)$. 
This property is already known for the case of the Heisenberg group
\cite[Theorem 5.5]{astengo_diblasio_ricci_08}.
The proof  relies on a generalisation 
\cite[Theorem 5.2]{astengo_diblasio_ricci_08}
of Hulanicki's Schwartz kernel Theorem
\cite{hulanicki_84}.

The converse inclusion has been recently shown 
for any Heisenberg Gelfand pair
\cite{astengo_diblasio_ricci_08}
and we prove it here
for $(SO(3)\ltimes \nf,SO(3))$ where $\nf$ is the free two-step nilpotent Lie
group with three generators;
we realise $\nf$ as $\Rb_x^3\times\Rb_y^3$,
$\{0\}\times \Rb_y^3$ being the centre.
It is known that $(SO(3)\ltimes \nf,SO(3))$ is a Gelfand pair
\cite[Theorem 5.12]{benson_jenkins_ratcliff_1990}.
We will give explicit formulae for a family of three essentially self-adjoint
operators $\Dc$ 
that generate ${\Db(\nf)}^{SO(3)}$,
the bounded spherical functions and their corresponding eigenvalues
for $\Dc$.

Our goal here is to prove that 
for any Schwartz $SO(3)$-invariant function $F\in {\Sc(\nf)}^{SO(3)}$,
there exists a Schwartz extension 
of its Gelfand transform:
$$
\mbox{i.e.}\quad
 \exists f\in \Sc(\Rb^3)
\qquad
f_{|\Sigma_{\Dc}}=\Gc F
.
$$

In the proof, we will use the known result for the three-dimensional
Heisenberg group $H_1$.
For this, 
let us consider $N'$, 
the quotient group of $\nf$ by the central subgroup~$\Rb^2_{(y_1,y_2)}$,
and $K'$,
the stabiliser of $\Rb^2_{(y_1,y_2)}$ in $SO(3)$.
We will see that
$N'$ is isomorphic to $H_1\times \Rb_{x_3}$,
and $K'$ is  isomorphic to $U_1\rtimes \Zb_2$
(see Section~\ref{sec_Gspec_N_N'}).
The Gelfand transform for the pair $(K'\ltimes N',K')$ will be denoted by $\Gc'$.

Whenever it makes sense,
we denote by~$\radon F$ the function on~$N'$ 
given by integration of a function $F$ of~$\nf$ 
on the central subgroup
$\Rb^2_{(y_1,y_2)}$.
The operator $\radon$ maps 
$SO(3)$-invariant functions on~$\nf$ to $K'$-invariant functions on~$N'$,
and Schwartz functions on~$\nf$ to Schwartz functions on~$N'$.
$\radon$ is 1-1, but does not send ${\Sc(\nf)}^{SO(3)}$ onto ${\Sc(N')}^{K'}$
(see Proposition~\ref{prop_im_radon}).
The definition of $\radon$ can be extended to left-invariant differential operators
in such a way that $\radon (D F)= (\radon D) (\radon F)$ 
for any left-invariant differential operators~$D$ 
and any smooth compactly supported functions~$F$ on~$N$.
We will see that 
the image of the operators in~$\Dc$ by~$\radon$ 
completed with~$-\partial_{x_3}^2$ 
gives a family~$\Dc'$ of essentially self-adjoint generators for~${\Db(N')}^{K'}$.
Again we will give explicit formulae for~$\Dc'$, 
the bounded spherical functions, 
and their corresponding eigenvalues for~$\Dc'$.

The spectrum of $(K'\ltimes N',K')$ can be projected 
onto the spectrum of $(SO(3)\ltimes \nf,SO(3))$
in the following sense:
composing an homomorphism of ${L^1(N')}^{K'}$ with $\radon$
provides a mapping 
$\Pi : \Sigma_{\Dc'} \rightarrow  \Sigma_\Dc$
between the two spectra, that is completely explicit here.
In fact  $\Pi$ maps continuously $\Sigma_{\Dc'}$ onto $\Sigma_\Dc$,
but is 1-1 only on the regular part of the spectrum
(see Section~\ref{sec_Gspec_N_N'}). 

For any Schwartz $SO(3)$-invariant function $F\in {\Sc(\nf)}^{SO(3)}$,
we have:
$$
\Gc' (\radon F) = \Gc F \circ \Pi 
.
$$
The existence of a Schwartz extension to $\Rb^4$ for $\Gc' (\radon F)$,
can be deduced easily from the Heisenberg case
\cite{astengo_diblasio_ricci_07,astengo_diblasio_ricci_08};
it does not imply directly the existence of a Schwartz extension for
$\Gc F$
 but is constantly used all along the proof.

This article is organised as follows.
In Section~\ref{sec_Gspec_N_N'}, 
we introduce the notations and the basic facts concerning
the Gelfand spectra of $(SO(3)\ltimes \nf,SO(3))$ and $(K'\ltimes N',K')$.
In Section~\ref{sec_results}, we give some general settings and the precise statements of our results.
In Section 4 we describe $\radon$ and the restriction mappings.
In Section 5 we give the proof of Theorem 5
using an extension of a mean value formula due to Geller 
in the case of the Heisenberg group \cite{geller_80}.
In the appendix, 
we give, for completeness, detailed proof of some results  appearing
in this paper and concerning differential operators and functional
calculus on them.

\section{The Gelfand spectra of $(SO(3)\ltimes \nf,SO(3))$ and
  $(K'\ltimes N',K')$}
\label{sec_Gspec_N_N'}

We realise $\nf$ as $\Rb_x^3\times\Rb_y^3$ endowed with the law:
$$
(x,y).(x',y')=(x+x',y+y'+\frac12 x\wedge x')
,
$$
where $\wedge$ indicates the usual wedge product in $\Rb^3$.
$\ncf$ denotes its Lie algebra.
For $j=1,2,3$, let $X_j$ be the left-invariant vector field on $N$ that equals
$\partial_{x_j}$ at 0,
and $Y_j$ the left-invariant vector field on $N$ that equals
$\partial_{y_j}$.
${(X_j)}_{j=1,2,3} $ and $ {(Y_j)}_{j=1,2,3}$ form the canonical basis of
$\ncf$,
and satisfy:
$$
[X_1,X_2]=Y_3
\quad,\quad
[X_3,X_1]=Y_2
\quad,\quad
[X_2,X_3]=Y_1
.
$$

The group~$SO(3)$ acts on~$\Rb^3$ and thus on~$\nf$
by acting simultanously on each copy of $\Rb^3$. 
One checks easily 
that this action is by automorphisms on $\nf$.
$(SO(3)\ltimes \nf,SO(3))$ is a Gelfand pair 
\cite[Theorem 5.12]{benson_jenkins_ratcliff_1990}.

Let us define the sub-Laplacian $L$,
the central Laplacian $\Delta$ and a third operator $D$ by:
$$
L=-\sum_{j=1}^3 X_j^2
\quad,\quad
\Delta=-\sum_{j=1}^3 Y_j^2 
\quad,\quad
D= -\sum_{j=1}^3 X_j Y_j
.
$$
In Section~\ref{sec_results},
we will show that these operators form a family
 $\Dc=(L,\Delta,D)$
of essentially self-adjoint operators that generate
${\Db(\nf)}^{SO(3)}$.

The bounded spherical functions 
and their corresponding eigenvalues
for $\Dc$ are known explicitly.
Let us define some notation first:
for any vector
$x=(x_1,x_2,x_3)\in\Rb^3$, 
we will write $\tilde x$ or $x\,\tilde{}$ for $(x_1,x_2)$ and, occasionally, ${[x]}_3$ for $x_3$.
$\Lc_l(u)=(1/l!)e^{u/2}(d/du)^lu^le^{-u}$ denotes the $l$-Laguerre function of order 0 on $\Rb$.
Then the bounded spherical functions on $N_{3,2}$ are:
$$
\phi_{\lambda,l,r}(x,y)
=\int_{k\in SO(3)}
e^{-i\lambda {[k.y]}_3}
\Lc_l\Big(\frac \lambda 2 \nn{ [k.x]\tilde{\,}}^2\Big)
e^{-ir{[k.x]}_3}
dk
\ ,\quad
\lambda>0,\, l\in\Nb,\,r\in\Rb.
$$
and
$$
\phi_{0,R}(x,y)=\int_{k\in SO(3)}e^{-iR{[k.x]}_3}\,dk
= \frac {\sin (R \nn{x})}{R \nn{x}}
dk
\ ,\quad
R\geq0;
$$
their eigenvalues for $\Dc$  are given by:
 \begin{eqnarray*}
\mu_{\phi_{\lambda,l,r}}
  &=&
  \left(\lambda(2l+1) +r^2,\lambda^2,\lambda r\right)
  ,\\
  \mu_{\phi_{0,R}}
  &=&
  \left(R^2, 0,0\right)
  .
\end{eqnarray*}
The Gelfand spectrum $\Sigma_\Dc$ of $(SO(3)\ltimes \nf,SO(3))$ is then realised as the union 
of the collection of the 
$\mu_{\phi_{\lambda,l,r}}$, $\lambda>0,\, l\in\Nb,\,r\in\Rb$
(the regular part of the spectrum),
with the collection of the $\mu_{\phi_R}$, $R\in\Rb$
(the singular part of the spectrum).
Calling $(\eta_1,\eta_2,\eta_3)$ 
the coordinates corresponding to $L,\Delta,D$ respectively, 
$\Sigma_\Dc$ is then the union, 
for $l\geq0$, of the surfaces $\Gamma_l$ 
defined by the equation
$\eta_3^2=\eta_2\big(\eta_1-(2l+1)\sqrt{\eta_2}\big)$. 
They all meet together on the positive $\eta_2$-axis, 
the singular part of $\Sigma_\Dc$.

$N'$ is the quotient group of $\nf$ by the central subgroup
$\Rb^2_{(y_1,y_2)}$;
we realise $N'$ as $\Rb^3_x\times\Rb_t$
endowed with the law:
$$
(x,t).(x',t')=(x+x',t+t'+\frac12 {[x\wedge x']}_3)
.
$$
$\Nc'$ denotes its Lie algebra; this is a quotient of $\ncf$ by $\Rb
Y_1\oplus\Rb Y_2$
and we denote $q:\ncf\rightarrow \Nc'$ the quotient mapping.
$q(X_j)=X'_j$ is the left-invariant vector field $X'_j$ on $N'$ that equals
$\partial_{x_j}$ at 0,  $j=1,2,3$;
$q(Y_1)=q(Y_2)=0$, 
and $q(Y_3)=T$ is the left-invariant vector field on $N$ that equals
$\partial_t$. In particular $X'_3=\partial_{x_3}$ lies in the centre of $\Nc'$.
${(X'_j)}_{j=1,2,3} $ and $T$ form the canonical basis of $\Nc'$.
It is easy to see that $N'$ is isomorphic to $H_1\times \Rb$.
Let $K'$ be the stabiliser of $\Rb^2_{(y_1,y_2)} \subset \Rb_y^3$ in $SO(3)$.
The group$K'$ is 
$S\big(O(2)\times O(1)\big)\cong O(2)\cong U_1\rtimes \Zb_2$.

$(K'\ltimes N',K')$ is a Gelfand pair 
and
its bounded spherical functions are explicitly known:
$$
\phi_{\lambda,l,r}'(x,t)
=
\cos(\lambda t +r x_{x_3})
\Lc_l(\frac \lambda 2 \nn{ {[k.x]}\tilde {\,}}^2)
dk
\quad,\quad
\lambda>0, l\in\Nb, r\in\Rb
,
$$
and
$$
\phi'_{\zeta,r}(x,y)
= J_o(\zeta \nn{\tilde x})\cos(r x_3)
\quad,\quad
\zeta,r\in\Rb
,
$$
$J_o$ being the Bessel function of order 0.

We define the following operators:
$$
L'=-\sum_{j=1}^3 {X'_j}^2
\quad,\quad
\Delta'= -T^2
\quad,\quad
D'= -X'_3 T
.
$$
The operators
$L'$, $\Delta'$, $D'$ and $-{X'_3}^2$
are $K'$-invariant, essentially self-adjoint 
and generate ${\Db(N')}^{K'}$
(see Proposition~\ref{prop_ess_selfadj}
and Subsection~\ref{subsec_N'}).
We set the family
$\Dc'=(L',\Delta', D', -{X'_3}^2)$.

The eigenvalues of the bounded spherical functions for~$\Dc'$  are given by:
 \begin{eqnarray*}
\mu_{\phi'_{\lambda,l,r}}
  &=&
  \left(\lambda(2l+1) +r^2,\lambda^2,\lambda r,r^2\right)
,\\
  \mu_{\phi'_{\zeta,r}}
  &=&
  \left(\zeta^2+r^2, 0,0, r^2\right)
.
\end{eqnarray*}
As in the case of $(SO(3)\ltimes\nf,SO(3))$,
the Gelfand spectrum $\Sigma'_{\Dc'}$ of $(K'\ltimes N',K')$ is then
realised as the union 
of a regular and a singular part:
the regular part is
the collection of the 
$\mu_{\phi'_{\lambda,l,r}}$, $\lambda>0,\, l\in\Nb,\,r\in\Rb$,
and the singular part is 
the collection of the $\mu_{\phi'_{\zeta,r}}$, $\zeta,r\in\Rb$.

With coordinates $(\eta_1,\eta_2,\eta_3,\eta_4)$ 
corresponding to $L',\Delta',D',-{X'_3}^2$ respectively, 
$\Sigma'_{\Dc'}$ is the union of the set
$\{(\eta_1,0,0,\eta_4):0\le\eta_4\le\eta_1\}$ (the singular set) 
and  the two-dimensional surfaces $\Gamma'_l$, $l\ge0$, 
defined by the system of equations 
$$
\left\{\begin{matrix}
\eta_3^2
=
\eta_2\big(\eta_1-(2l+1)\sqrt{\eta_2}\big)\\ 
\eta_3^2
=
\eta_2\eta_4\ .\hfill 
\end{matrix}\right.
$$

Notice that the projection onto the hyperplane $\eta_4=0$ 
parallel to the $\eta_4$-axis maps $\Sigma'_{\Dc'}$ onto $\Sigma_\Dc$, 
and is bijective between the two regular sets. 
This fact, alluded to already in the introduction, 
will be relevant in view of the mapping $\radon$ defined in Subsection 4.1.

Let us give an equivalent and intrinsic point of view of this fact.
As explained in the introduction,
the spectrum of $(K'\ltimes N',K')$ can be projected in the following sense 
onto the spectrum of $(SO(3)\ltimes \nf,SO(3))$:
the composition of an homomorphism of ${L^1(N')}^{K'}$ with $\radon$
(see also Subsection~\ref{subsec_radon})
provides a mapping 
between the two Gelfand spectra.
Realising the Gelfand spectra as explained in the introduction
(see also Section~\ref{sec_results}),
this mapping
$\Pi : \Sigma_{\Dc'} \rightarrow  \Sigma_\Dc$
is then given by:
\begin{eqnarray*}
   \Pi \left(\lambda(2l+1)+r^2,\lambda^2,\lambda r,r^2\right) 
& = &
  \left(\lambda(2l+1)+r^2,\lambda^2,\lambda r\right)
,
  \\
\Pi  \left(\zeta^2+r^2,0,0,r^2\right)
& = &
  \left(\zeta^2+r^2,0,0\right)  
.  
\end{eqnarray*}
$\Pi$ maps continuously
$\Sigma_{\Dc'}$ onto $\Sigma_\Dc$.
Moreover
$\Pi$ maps homeomorphically
the regular part of $\Sigma_{\Dc'}$ onto the regular part of $\Sigma_\Dc$;
$\Pi$ maps  the irregular part of $\Sigma_{\Dc'}$ onto the
irregular part of $\Sigma_\Dc$, but this correspondence is not 1-1. 

\section{Results}
\label{sec_results}

In this section, we describe the general settings of our work 
and explain the conjecture $\Gc ({\Sc(N)}^K)=\Sc(\Sigma_\Dc)$.
We will give the precise statement of our main result in Theorem~\ref{thm_goal1}.

Let $N$ be a connected, simply-connected Lie group,
$\Nc$ its Lie algebra,
$\exp:\Nc\rightarrow N$ the exponential mapping 
and
${(E_i)}_{i=1}^p$ a basis of~$\Nc$.
The canonical basis ${(E_i)}_{i=1}^p$ of~$\Nc$ being chosen, 
this induces a Lebesgue measure $dX$ on $\Nc$ 
and, via the exponential map, a Haar measure $dn$ on $N$;
the spaces $L^p(N)$ are defined with respect to this Haar measure.
When $N$ is a graded Lie group,
following~\cite[ch1.D]{folland_stein_bk},
we fix a homogeneous gauge $\nn{.}$ on $N$ and 
we keep the same notation for the basis $(E_j)$ of $\Nc$ 
and the associated left-invariant vector fields on $N$;
we set the following family of semi-norms 
parametrised by $a\in \Nb$
on the Schwartz space $\Sc(N)$ 
which induces the usual Fr\'echet space structure on~$\Sc(N)$:
$$
\nd{F}_{a,N}
=
\sup_{n\in N,  d(I) \leq a}
{\left(1+\nn{n}\right)}^a \nn{E^I F(n)}
.
$$

$\Pc(\Nc)$ denotes the algebra of polynomials on~$\Nc$ with real coefficients
where $\Nc$ is then identified with the Euclidean vector space~$\Rb^p$.
$\Db(N)$ denotes the algebra of real left-invariant differential operators on~$N$,
as operators defined on $C_c^\infty(N)$,
the space of smooth, compactly-supported functions on~$N$.

To $P\in \Pc(\Nc)$, 
we associate $D_P\in \Db(N)$ by:
$$
D_P F(n)={\left[P(i^{-1}\partial_u)F(n\exp (\sum_{j=1}^p u_j E_j))\right]}_{|u=0}
.
$$
We obtain the symmetrisation mapping
$P\mapsto D_P$,
that is a linear isomorphism 
between the algebras~$\Db(N)$ and~$\Pc(\Nc)$  
\cite[Ch.II Theorems 4.3 and 4.9]{helgason_bk2}

In the appendix we show:
\begin{prop}
\label{prop_ess_selfadj}
If $(K\ltimes N,K)$ is a Gelfand pair,
each operator of~${\Db(N)}^K$ is essentially self-adjoint, 
that is, 
it admits a unique self-adjoint extension to an unbounded operator of~$L^2(N)$.

Furthermore the operators of~${\Db(N)}^K$ commute strongly, 
in the sense that 
the spectral resolutions of their self-adjoint extensions commute.
 \end{prop}
We will use the same notation for an operator of~${\Db(N)}^K$
and its self-adjoint extension.

By Hilbert's Basis Theorem,
if a group~$K$ acts orthogonally on some Euclidean space $\Rb^p$, 
the algebra~${\Pc(\Rb^p)}^K$ of $K$-invariant polynomials on~$\Rb^p$ 
is finitely generated 
\cite[Ch.II Corollary 4.10]{helgason_bk2}.
If $\rho_1,\ldots,\rho_q$ is a set of generators,
we call $\{\rho_1,\ldots,\rho_q\}$ a Hilbert basis for~$(\Rb^p,K)$
and $\rho=(\rho_1,\ldots,\rho_q)$ the corresponding Hilbert mapping.
Furthermore,
if $\rho=(\rho_1,\ldots,\rho_q)$ and $\rho'=(\rho'_1,\ldots,\rho'_{q'})$ 
are two Hilbert mappings for $(\Rb^p,K)$,
then there exists $Q=(Q_1,\ldots,Q_q)$,
$Q_j\in \Pc(\Rb^{q'})$,
such that $\rho = Q \circ \rho'$ (and viceversa).
We will make extensive use of G.~Schwarz's Theorem \cite{schwarz_1975}:
every $K$-invariant smooth function on $\Rb^p$
can be expressed as a smooth function of a Hilbert basis $\rho$ of
$(\Rb^p,K)$.
In other words, 
the Hilbert map, $\rho$, induces  
an application
given by $\rho^*(h)=h\circ \rho$.
Moreover $\rho^*$ is a linear continuous mapping
from $C^\infty(\Rb^q)$ onto ${C^\infty(\Rb^p)}^K$,
and also from $\Sc(\Rb^q)$ onto ${\Sc(\Rb^p)}^K$
\cite[Theorem 6.1]{astengo_diblasio_ricci_08}.

Assume that $(K\ltimes N,K)$ is a Gelfand pair.
Any family of generators of~${\Db(N)}^K$ is obtained as the symmetrisation of a Hilbert basis,
and conversely, 
if $\{\rho_1,\ldots,\rho_q\}$ denotes a Hilbert basis for $(\Nc,K)$,
then $\{D_{\rho_1},\ldots,D_{\rho_q}\}$ is a set of generators of ${\Db(N)}^K$.

Let us fix $(\rho_1,\ldots,\rho_q)$ an ordered Hilbert basis
for~$(\Nc,K)$, 
to which we associate the ordered family of operators
$\Dc_\rho=(D_{\rho_1},\ldots,D_{\rho_q})$.
We denote by~$\Sigma_{\Dc_\rho}$,
the set of the $q$-tuples of eigenvalues 
$\mu(\phi)=(\mu_1(\phi),\ldots, \mu_q(\phi))$
of~$\Dc_\rho$ 
for the bounded $K$-spherical functions $\phi$ on $N$.
As mentioned in the introduction and proved in~\cite{ferrari},
$\Sigma_{\Dc_\rho}$ is 
the realisation of the Gelfand spectrum associated to $\Dc_\rho$,
in the sense that
the set  $\Sigma_{\Dc_\rho}$ of such $q$-tuples is in 1-1 correspondence with
the Gelfand spectrum and the topology induced on it from $\Rb^q$
coincides with the Gelfand topology. 
In the appendix, we will show that $\Sigma_{\Dc_\rho}$
is also the joint spectrum of 
~$\Dc_\rho$:
\begin{prop}
\label{prop_joint_spectrum}
Let $(\rho_1,\ldots,\rho_q)$ be an ordered Hilbert basis for $(\Nc,K)$.
The joint spectrum of the family of strongly commuting, self-adjoint
operators~$\Dc_\rho=(D_{\rho_1},\ldots,D_{\rho_q})$ is~$\Sigma_{\Dc_\rho}$.
\end{prop}

For a closed subset~$E$ of~$\Rb^q$, 
$\Sc(E)$ denotes the space of restrictions to~$E$ of Schwartz functions, 
endowed with the quotient topology of $\Sc(\Rb^q) / \{ f: f_{|E}=0\}$;
we will often define a class in this quotient as 
being given as the restriction of a Schwartz function on $\Rb^q$.
The spectrum $\Sigma_\Dc$ is a closed subset of $\Rb^q$.
We are interested in the conjecture
${\Sc(N)}^K\overset {\Gc}{\sim}\Sc(\Sigma_\Dc)$.
The existence of a polynomial mapping  between two Hilbert mappings
implies that 
the validity of this conjecture is independent of the choice of~$\Dc$
(see \cite[Section~3]{astengo_diblasio_ricci_08}).
The continuous inclusion 
$\Sc(\Sigma_\Dc) \hookrightarrow \Gc ({\Sc(N)}^K)$
relies on the following statement,
which is a generalisation of Hulanicki's Schwartz Kernel Theorem
proved in the appendix:
\begin{prop}
\label{prop_gen_hula}
  Let $(\rho_1,\ldots,\rho_q)$ be an ordered Hilbert basis for~$(\Nc,K)$,
and $\Dc_\rho=(D_{\rho_1},\ldots,D_{\rho_q})$ the associated family of
operators.

Let  $m$ be in~$\Sc(\Rb^q)$.
The operator~$m(\Dc_\rho)$ 
is a convolution operator 
with a~$K$-invariant Schwartz kernel 
$M=M_{m,\Dc_\rho}\in{\Sc(N)}^K$:
$$
\forall F\in L^2(N)
\qquad
m(\Dc_\rho)F= F*M
.
$$
The Gelfand transform of $M$ is:
$$
\Gc M=m_{|\Sigma_{\Dc_\rho}}
.
$$

Furthermore the mapping 
$m\in\Sc(\Rb^q)\mapsto M_{m,\Dc_\rho}\in{\Sc(N)}^K$
is continuous.
\end{prop}

For the Gelfand spectra of Heisenberg groups 
or the free two-step nilpotent Lie groups, 
the inclusion of the spectrum in the image of the Hilbert mapping:
$$
\Sigma_{\Dc_\rho}\subset \im \rho
,
$$ 
is true,
independently of the choice of the Hilbert mapping $\rho$
(but we do not know if it is true in general).
Here we will use this property only in the case of $N'=H_1\times\Rb$, 
where it is known, 
the spectrum and the Hilbert mapping being explicit.

\begin{lem}
The polynomials
$\nn{x}^2$, $\nn{y}^2$ and $x \cdot y$ 
generate the algebra of polynomials on $\Rb^3_x\times \Rb^3_y$ 
that are invariant under the simultaneous action of $SO(3)$ on each copy of~$\Rb^3$.
\end{lem}

\begin{proof}
If $P(x,y)$ is an $SO(3)$-invariant
polynomial on $\Rb^3_x\times \Rb^3_y$,
then for each independent vectors $x,y\in \Rb^3$,
we have $P(x,y)=P(-x,-y)$ 
because the linear transformation 
that equals $-\id$ on the vector space spanned by $x$ and $y$, 
and $1$ on the orthogonal complement line, 
is in $SO(3)$;
this shows that $P$ is invariant under $-\id_{\Rb^3}$,
and thus also under the simultaneous action of $O(3)$ on each copy
of~$\Rb^3$.
This implies:
$$
{\Pc(\Rb_x^3\times\Rb_y^3)}^{SO(3)}
=
{\Pc(\Rb_x^3\times\Rb_y^3)}^{O(3)}
.
$$
By \cite[Theorem~4.2.2.(1)]{goodman_wallach_bk},
${\Pc(\Rb_x^3\times\Rb_y^3)}^{O(3)}$ is spanned 
by $\nn{x}^2$, $\nn{y}^2$ and $x \cdot y$.
\end{proof}

Thus $\rho(x,y)=(\nn{x}^2,\nn{y}^2, x \cdot y)$ 
gives a Hilbert mapping for $(\ncf,SO(3))$. 
We compute easily that the associated family of operators
by symmetrisation
is $\Dc=(L,\Delta,D)$ defined in Section~\ref{sec_Gspec_N_N'},
where we give also an explicit description
of the associated realisation of the Gelfand spectrum.

\begin{thm}
  \label{thm_goal1}
The Gelfand transform of any Schwartz $SO(3)$-invariant function on $\nf$
admits a Schwartz extension to $\Rb^3$: 
$$
\forall F\in {\Sc(\nf)}^{SO(3)}
\qquad
\Gc F\in \Sc(\Sigma_{\Dc})
.
$$
Moreover the mapping 
$F\in {\Sc(\nf)}^{SO(3)}\mapsto \Gc F\in \Sc(\Sigma_{\Dc})$
is an isomorphism of Fr\'echet spaces.
\end{thm}

The group $N_{3,2}$ admits a slightly bigger group of automorphisms
than $SO(3)$, 
namely $O(3)$ acting by:
$$
k(x,y)=\big(kx,(\det k) ky\big)
\quad,\quad
k\in O(3)
,
$$

It is easily verified that $\{\nn{x}^2,\nn{y}^2, {(x \cdot y)}^2\}$ 
gives a Hilbert basis for $(\ncf,O(3))$ 
and the associated family of operators
by symmetrisation
is $\tilde{\mathcal D}=(L,\Delta,D^2)$. 
Following the same lines as in
\cite[Section~8]{astengo_diblasio_ricci_08}, 
we have the following.

\begin{cor}
The Gelfand transform is an isomorphism
between ${\Sc(\nf)}^{O(3)}$ and $\Sc(\Sigma_{\tilde{\mathcal D}})$
as Fr\'echet spaces.
\end{cor}

From now on,  
$N$ will stand for $\nf$ and $K$ for $SO(3)$.

\section{$\radon$ and restriction mappings}
\label{sec_mappings}

\subsection{The mapping $\radon$}
\label{subsec_radon}

In the introduction, 
we denoted by~$\radon F$ the function on~$N'$ 
given by integration of a function $F$ of~$N$ on 
the central subgroup
$\Rb^2_{(y_1,y_2)}$
whenever it makes sense,
for example on $L^1(N)$. 
It is sometimes convenient to consider $\radon$ 
as acting between functions defined on the Lie algebras, 
rather than on the groups. 
We will do so without any further mention.
The operator $\radon$ maps $K$-invariant functions on~$N$ 
to $K'$-invariant functions on~$N'$,
integrable functions on $N$ to integrable functions on $N'$ 
continuously,
Schwartz functions on~$N$ to Schwartz functions on~$N'$
continuously.
It respects convolution on the groups and abelian convolution on the Lie algebras.

We extend the definition of $\radon$ to left-invariant
differential operators on~$N$.
Let $D=D_P\in \Db(N)$, $P\in \Pc(\Nc)$.
Easy changes of variables,
see e.g. (\ref{eq1_subsec_proof_selfadj}) below,
show:
$$
\forall F\in C_c^\infty(N)
\quad,\quad
\forall G\in C_c^\infty(N')
\qquad
\langle \radon (DF), G \rangle
=
\langle \radon F, D_{P_{|\Nc'}}G \rangle
.
$$
We are led to define $\radon:\Db(N)\rightarrow \Db(N')$ by:
$\radon D_P=D_Q$
where $Q=P_{|\Nc'}$ is the restriction of $P$ to $\Nc'$.
For any $D\in \Db(N)$,
$\radon D$ is the only differential operator $T$ on $N$  satisfying:
$$
\forall F\in C_c^\infty(N)
\qquad
\radon [D F]=T \radon F
.
$$

The mapping $\radon$ on functions 
is dual to the restriction mapping from $\Nc$ to $\Nc'$
in the following sense.
Let us denote $\Fc_y$ and $\Fc_t$ the Fourier transform with respect
to the variables $y\in\Rb^3$ and $t\in\Rb$ respectively given by:
\begin{eqnarray*}
  \Fc_y F(x, \hat y)
  &=&
  \int_{\Rb^3} F(x, y) e^{-i y.\hat y} dy
,
  \\  
  \Fc_t G(x, \hat t)
  &=&
  \int_\Rb G(x, t) e^{-i t.\hat t} dt
;
\end{eqnarray*}
whenever it makes sense for a function $G$ on $N'$ and a function $F$ on $N$,
identified with functions on $\Nc'$ and $\Nc$ respectively,
we have:
\begin{equation}
  \label{eq_radon_restrict}
G=\radon F
\Longleftrightarrow 
\Fc_t G =  \Fc_y F_{|\Nc'}   
\end{equation}

In the following subsection,
we describe the restriction mapping.

\subsection{Restriction and radialisation mappings}

For a function $f$ on $\Nc$,
we denote by $\restrict f=f_{|\Nc'}$ its restriction to $\Nc'$.
We set:
$$
\Nc_o=\Rb^3_x\times (\Rb^3_y\backslash \{0\})
\quad\mbox{and}\quad
\Nc'_o=\Rb^3_x \times (\Rb_t\backslash \{0\})
.
$$

In the next lemma, we define the radialisation mapping $\radial$:
\begin{lem}
\label{lem_radial}
 For a function $h \in {C^\infty(\Nc'_o)}^{K'}$
and $(x,y)\in \Nc_o$,
the following: 
$$
\radial(h)(x,y)=h(k^{-1}x, t)
\quad,\quad
\mbox{where} 
\quad
y=k(0,0,t)\;
\mbox{for some}\; k\in K.
$$
defines a $K$-invariant function  $\radial(h)$,
that is  smooth on $\Nc_o$.

$\radial$ is an isomorphism between the topological vector spaces
${C^\infty(\Nc_o)}^K$ and ${C^\infty(\Nc'_o)}^{K'}$,
whose inverse is $\restrict$.
\end{lem}

\begin{proof}
 For a function $h \in {C^\infty(\Nc'_o)}^{K'}$
and $(x,y)\in \Nc_o$,
it is easy to see that $\radial(h)(x,y)$ is well defined
and $K$-invariant.
  
Let us show that $\radial(h)\in {C^\infty(\Nc_o)}^K$.
We choose a basis ${(A_j)}_{j=1,2,3}$ for the Lie algebra of $K$.
At a point $(x_0,y_0)\in\Nc_0$ (with $y_0=k_0(t_0,0,0)$, $t_0=|y_0|\ne0$) we choose a local coordinate system $(x,y)=\big(x,k(t,0,0)\big)=\sigma(x,k,t)$, where $x\in\Rb^3$, $t\in\Rb^+$ and $k$ varies in a small two-dimensional surface in $K$ containing $k_0$ and transversal to $k_0K'$. This change of variables does not affect the derivatives in $x$, whereas
$$
\partial_{y_j}=
c_{j,0}(k,t) \partial_t
+\sum_{j'=1,2,3} c_{j,j'} (k,t) A_{j'}
.
$$
By homogeneity,
$c_{j,0}\in C^\infty(K_k\times \Rb_t^*)$ is homogeneous of degree 0 in $t$,
and the $c_{j,j'}\in C^\infty(K_k\times \Rb_t^*)$ homogeneous of degree $(-1)$ in $t$.
More generally,
we can write the derivative
$$
\partial_y^I=\partial_{y_1}^{i_1}
\partial_{y_2}^{i_2}
 \partial_{y_3}^{i_3}
\quad,\quad
I=(i_1,i_2,i_3)\in\Nb^3,
$$
as:
\begin{equation}
  \label{eq_partial_y_I}
\partial_y^I=
\sum c_{I,J}(k,t) \partial_t^{j_0}
 A_1^{j_1}  A_2^{j_2} A_3^{j_3}
,  
\end{equation}
where the sum is over $J=(j_0,j_1,j_2,j_3)\in\Nb^4$,
with $\nn{J}=\nn{I}$,
and the $c_{I,J}\in C^\infty(K_k\times \Rb_t^*)$ are homogeneous of
degree~$(j_0-\nn{I})$  in~$t$.
As the function $(x;k,t)\mapsto h(k^{-1}x, t)$ is smooth on $\Rb^3 \times K \times \Rb$,
(\ref{eq_partial_y_I}) implies that $\radial h$ is smooth on $\Nc_o$.
Furthermore $h\in C^\infty(\Nc'_o)\mapsto \radial h\in C^\infty(\Nc_o)$ is continuous.
\end{proof}

Lemma~\ref{lem_radial} implies that the mapping $\restrict$ is a 1-1 on 
${C^\infty(\Nc)}^K$ and on ${\Sc(\Nc)}^K$.
Let us determine $\restrict ({C^\infty(N)}^K)$.
We will need the following notation:
\begin{itemize}
\item For $f\in {C^\infty (\Nc)}^K$,
  we denote by $P_M^{(f)}(x,t)$ 
  the homogeneous polynomial of degree $M$
  in the Taylor expansion of $f(x,\cdot)$  at $y=0$:
  $$
  P_M^{(f)}(x,y)
  =
  \sum_{\nn{j}= M} \frac 1{j!}\partial_y^j f(x,0) y^j
.
  $$
\item For $g\in {C^\infty (\Nc')}^{K'}$,
  we denote by $Q_M^{(g)}(x,t)$ 
  the homogeneous polynomial of degree $M$
  in the Taylor expansion of $g(x,\cdot )$  at $t=0$:
  $$
  Q_M^{(g)}(x,t)
  =
  \frac 1{M!}\partial_t^M g(x,0) t^M
.
  $$
\end{itemize}

We see:
$$
Q_M^{(\restrict f)}
=
\restrict P_M^{(f)} 
\quad\mbox{and thus}\quad
\radial Q_M^{(\restrict f)}
=
P_M^{(f)}
.
$$
Thus a function $g\in   {C^\infty (\Nc')}^{K'}$ 
that is the restriction of some
function $f\in   {C^\infty (\Nc)}^K$,
necessarily has the following property:

\textbf{Property (R).}
\textit{  For any $M\in\Nb$,
  $\radial(Q_M^{(g)})$ extends to a smooth function on $\Nc$ 
  which is a homogeneous polynomial in $y$ of degree $M$,
  with smooth coefficients in $x$. }

It turns out that this condition is also sufficient:

\begin{prop}
\label{prop_im_restrict}
  Let $g\in   {C^\infty (\Nc')}^{K'}$.

  The function $g$ is in the image of $\restrict$ 
  if and only if it satisfies Property~(R).

  In this case,
  $\radial (g)$ extends to a $K$-invariant smooth function $f$ on $\Nc$,
  whose restriction to $\Nc'$ is $g$
  and we have $Q_M^{(g)}=\restrict [P_M^{(f)}]$.
Moreover if in addition $g\in {\Sc(\Nc')}^{K'}$,
then $\radial (g)\in {\Sc(\Nc)}^{K}$
\end{prop}

In the proof, we adapt the ideas of the Euclidean setting
\cite[Theorem~2.4]{helgason_bk3}.

\begin{proof}
  Let $g\in   {C^\infty (\Nc')}^{K'}$ satisfying
  Property~(R).
  For each $M$,
  we denote $P_M$ the extension of 
  $\radial(Q_M^{(g)})$
  to a smooth function on $\Nc$ 
  that is a homogeneous polynomial in $y$ of degree $M$,
  with smooth coefficients in $x$.

Let $M_o\in\Nb$.
The Taylor Formula gives:
\begin{equation}
  \label{eq_taylor_formula}
g(x,t)-\sum_{j=0}^{M_o} Q_j^{(g)} (x,t)
=
\frac{t^{M_o+1}}{M_o!} 
\int_0^1 {(1-w)}^{M_o} 
\left( \partial_t^{M_o+1} g\right) (x, wt)
dw
.  
\end{equation}
Let $I_o\in\Nb^3$ with $\nn{I_o}=M_o+1$.
We have on $\Nc_o$:
$$
\partial_y^{I_o} \radial(g)
=
\partial_y^{I_o} \left(\radial(g)-\sum_{j=0}^{M_o}P_j \right)
=
\partial_y^{I_o} \left[\radial\left (g-\sum_{j=0}^{M_o} Q_j^{(g)} \right)\right]
.
$$
Now for any $(x,y)\in \Nc_o$,
$y\not= 0$ can be written $y=k(0,0,t)$, $t\in\Rb^*$, $k\in K$,
and by (\ref{eq_partial_y_I}) and (\ref{eq_taylor_formula}),
$\left(\partial_y^{I_o} \radial(g)\right)(x,y)$
can be written as the sum over $J\in\Nb^4$,
$\nn{J}=M_o+1$,
of:
$$
\frac{c_{I_o,J}(k,t) }{M_o!} 
\int_0^1 {(1-w)}^{M_o} 
\partial_t^{j_0}
 A_1^{j_1}  A_2^{j_2} A_3^{j_3}
\left[t^{M_o+1} \left( \partial_t^{M_o+1} g\right) (k^{-1}x, wt)\right]
dw
.
$$
This last term remains bounded if $0<\nn{y}=\nn{t}\leq 1$
because $c_{I_o,J}$ is homogeneous of degree $j_o-(M_o+1)$.
This implies that $\partial_y^{I_o} \radial(g)$ is bounded on a compact
neighborhood of $(x,0)$ for any $x$, and any $I_o$ and $M_o$.
It is easy to see that for any $I\in\Nb^3$,
 $\partial_x^{I}\partial_y^{I_o} \radial(g)$ satisfies the same
 conditions.
Local boundedness of all derivatives is sufficient to imply that
$\radial (g)$ has a smooth extension to $\Nc$.
\end{proof}

We deduce easily from (\ref{eq_radon_restrict})
 the following characterisation of $\radon({\Sc(\Nc)}^K)$:
\begin{prop}
\label{prop_im_radon}
 Let $G\in   {\Sc (N')}^{K'}$.
  The function $G$ is in $\radon ({\Sc(N)}^K)$ 
  if and only if either of the following equivalent conditions is satisfied:
  \begin{itemize}
\item[{\rm (i)}]  $\Fc_t G$ 
 (identified with a function on $\Nc'$) 
 satisfies 
  Property~(R);
\item[{\rm (ii)}] denoting by $\Fc_{x,t}$ the Fourier transform with respect
to the variables $x\in\Rb^3$ and $t\in\Rb$ given by:
$$  
\Fc_{x,t} G(\hat x, \hat t)
  =
  \int_{\Rb^3\times \Rb } G(x, t) e^{-i x.\hat x} e^{-i t.\hat t} dx dt
\ ,
$$
 $\Fc_{x,t} G$ 
 (identified with a function on $\Nc'$) 
 satisfies   Property~(R);
\end{itemize}
\end{prop}

From G. Schwarz's Theorem,
it follows
(compare with \cite[Theorem~2.4]{helgason_bk3}):

\begin{cor}
Let $G\in   {\Sc (N')}^{K'}$.
  The function $G$ is in $\radon ({\Sc(N)}^K)$ 
  if and only if either of the following equivalent conditions is satisfied:
  \begin{itemize}
\item[{\rm (i)}]
for each $j\in\Nb$
  there exist Schwartz functions $a_{j,i}\in \Sc(\Rb)$, $i=0,\ldots,j$
  satisfying:
  $$
  \forall x\in \Rb^3
  \qquad
  \int_\Rb G(x,t)t^j dt=\sum_{i=0}^j a_{i,j}(\nn{x}^2)
  x_3^i\ ;
  $$
\item[{\rm (ii)}] for each $j\in\Nb$
there exist Schwartz functions $b_{j,i}\in \Sc(\Rb)$, $i=0,\ldots,j$
satisfying:
$$
\forall \zeta\in \Rb^3
\qquad
\int_{\Rb^3}\int_\Rb G(x,t)t^j e^{-i x.\zeta}dt dx
=
\sum_{i=0}^j b_{i,j}(\nn{\zeta}^2)\zeta_3^i
.
$$
\end{itemize}  
\end{cor}

\section{Proof of Theorem~\ref{thm_goal1}}
\label{sec_proofs}

Here we give the proof of Theorem~\ref{thm_goal1}.
It is based on the properties of mappings explained 
in Section~\ref{sec_mappings}
and on results already shown on the Heisenberg group
\cite{astengo_diblasio_ricci_07,astengo_diblasio_ricci_08}.
These two key ingredients are used 
in the proofs of a ``Geller-type'' Lemma  (Subsection~\ref{subsec_geller_lem})
and  of Theorem~\ref{thm_goal1} (Subsection~\ref{subsec_proof_thmgoal}).

\subsection{The Gelfand pair $(K'\ltimes N',K')$}
\label{subsec_N'}

We easily check that
$$
\rho'(x,t)=(\nn{x}^2,t^2,x_3 t, x_3^2)
,
$$
defines a Hilbert mapping of $(\Nc',K')$,
which satisfies:
$$
\rho'(x,t)=\left(\rho_{|\Nc'}(x,(0,0,t)),x_3^2\right)
\quad\mbox{and}\quad
\Dc'=\Dc_{\rho'}
.
$$

From the Heisenberg case
\cite{astengo_diblasio_ricci_07,astengo_diblasio_ricci_08},
we deduce:
\begin{lem}
\label{lem_heis_case}
For any $G\in {\Sc(N')}^{K'}$, 
there exists $\tilde g\in \Sc(\Rb^4)$ 
with $\Gc'G=\tilde g_{|\Sigma_{\Dc'}}$.

Furthermore the mapping
$G\in {\Sc(N')}^{K'}\mapsto\Gc' G\in \Sc(\Sigma_{\Dc'})$
is continuous.
\end{lem}

Precisely, continuity of the last mapping means that
\begin{eqnarray}
  &\forall a\in\Nb
\quad\exists C=C(a)>0
\quad\exists a'\in\Nb
\quad \forall G\in {\Sc(N')}^{K'}
&\nonumber \\
&\exists \tilde g\in \Sc(\Rb^4)
\quad \tilde g_{|\Sigma_{\Dc'}}=\Gc' G
\quad \nd{\tilde g}_{a,\Rb^4} \leq C \nd{G}_{a',N'}
.&\label{lem_heis_case_cont.}
\end{eqnarray}
Notice that the extension $\tilde g$ depends on the Schwartz semi-norm
$\nd{.}_{a,\Rb^4}$.

\subsection{The Geller-type Lemma}
\label{subsec_geller_lem}

In this subsection, we will state and prove a ``Geller-type Lemma'',
extending \cite{geller_80,astengo_diblasio_ricci_08}.
For this purpose
we will need the following remark.

Let $F\in {\Sc(N)}^K$.
The mapping
$$
R\mapsto \Gc F(R^2,0,0)=
\int F(x,y) e^{-i Rx_1} dx \,dy
,
$$
is a Schwartz even function on $\Rb$;
by Whitney's Theorem, 
 there exists a Schwartz function $f_o\in \Sc(\Rb)$ such that
$$
\forall R\in\Rb\qquad
f_o(R^2)=\Gc F(R^2,0,0)
;
$$
by Hulanicki's Schwartz Kernel Theorem 
or Proposition~\ref{prop_gen_hula},
$f_o(L)$  is a convolution operator with
a $K$-invariant Schwartz kernel 
which we denote by $\Gc F(L,0,0)$
(for brevity reasons, 
in this section 
we will often denote a convolution operator and its kernel by the same symbol).

\begin{prop}[Geller-type Lemma]
  \label{prop_Glem} 
Let $F\in {\Sc(N)}^K$.
There exist 
$F_1\in {\Sc(N)}^K$ 
and 
$F_2\in {\Sc(N)}^K$
satisfying:
  $$
  F
  = \Gc(F)(L,0,0) + \Delta F_1 +D F_2
.
  $$
\end{prop}

\begin{proof}
Let $F$ be in ${\Sc(N)}^K$ and $G=\radon F\in {\Sc(N')}^{K'}$.
By Lemma~\ref{lem_heis_case}
there exists 
$\tilde g\in \Sc(\Rb^4)$
with $\tilde g_{\sigma_{\Dc'}}=\Gc' G$.
By Proposition~\ref{prop_gen_hula},
the operator given by:
$$
\int_{w=0}^1 \partial_2 \tilde g\left(L, w\Delta,0,0\right)dw
,
$$
is a convolution operator with a $K$-invariant Schwartz kernel
which we denote by $F_1\in {\Sc(N)}^K$.
By spectral calculus,
we have : 
$$
\Delta F_1= \tilde g\left(L,\Delta,0,0\right)-\tilde g\left(L,0,0,0\right)
.
$$

 We will have finished the proof of  Proposition~\ref{prop_Glem}
once we have shown:
\begin{equation}
  \label{eq_exists_F2}
\exists F_2\in {\Sc(N)}^K
\qquad
  F-\Gc F\left(L,0,0\right)-\Delta F_1 = D F_2
\end{equation}

We denote by $H\in {\Sc(N)}^K$ and $I\in {\Sc(N')}^{K'} $ 
the functions given by:
  \begin{eqnarray*}
    H
&=&
F-\Gc F\left(L,0,0\right)-\Delta F_1
    =F-\tilde g\left(L,\Delta,0,0\right)
,\\
    I
&=&
\radon H
=G- \tilde g\left(L',\Delta',0,0\right)
.
  \end{eqnarray*}
  The Gelfand transform of $I$ is given by:
  \begin{equation}
    \label{eq_GT_I}
  \Gc' I \left(\mu_{\phi'}\right)=
  \Gc' G \left(\mu_{\phi'}\right)
  -\tilde g\left(L'\phi'(0),\Delta'\phi'(0),0,0\right)
.  
  \end{equation}
On the singular part of the spectrum, 
(\ref{eq_GT_I}) yields to:
  $$
  \Gc' I \left(\mu_{\phi'_{\zeta,r}}\right)= 
  \tilde g\left(\zeta^2+r^2,0,0,r^2\right)- \tilde g\left(\zeta^2+r^2,0,0,0\right)= 0
,
  $$
  because 
  $\tilde g\left(\zeta^2+r^2,0,0,r^2\right)=\tilde g\left(\zeta^2+r^2,0,0,0\right)$ 
  as $\Gc' G=\Gc F\circ \Pi $;
  this implies:
  \begin{equation}
    \label{eq:I_x_zero}
    \forall x\in \Rb^3\qquad
    \int_\Rb I(x,t) dt =0
.  
  \end{equation}
On the regular part of the spectrum, 
 (\ref{eq_GT_I}) yields to:
  $$
  \Gc' I \left(\mu_{\phi'_{\lambda, l,r}}\right)=
  \tilde g\left(\lambda(2l+1)+r^2, \lambda^2,\lambda r,r^2\right)
  -\tilde g\left(\lambda(2l+1)+r^2, \lambda^2,0,0\right)
,
  $$
  and in particular for $r=0$:
  $$
  \Gc' I \left(\mu_{\phi'_{\lambda, l,0}}\right)=
  \tilde g\left(\lambda(2l+1), \lambda^2,0,0\right)
  -\tilde g\left(\lambda(2l+1), \lambda^2,0,0\right)
  =0
  ;
  $$
  this implies for all $\lambda>0$:
  $$
  \forall l\in\Nb
  \qquad
  \int_{N'} I(x,t) e^{-i\lambda t} \Lc_l\left(\frac \lambda 2 \nn{ x'}^2\right)
  dx dt
  =0
,
  $$
  and $\{\Lc_l\}_{l\in\Nb}$ being an orthogonal basis of $L^2(\Rb^+)$,
  $$
  \forall \tilde x\in\Rb^2
  \qquad
  \int_{\Rb^2} I(\tilde x,x_3;t) e^{-i\lambda t} dx_3 dt
  =0
.
  $$
  Eventually, we get:
  \begin{equation}
    \label{eq:I_x't_zero}
    \forall \tilde x\in\Rb^2
    \quad,\quad
    \forall t\in\Rb
    \qquad
    \int_{\Rb^2} I(\tilde x,x_3;t) dx_3 
    =0
.  
  \end{equation}

  Let us set:
  $$
  G_2(x,t)
  =
  \int_{-\infty}^{x_3}
  \int_{-\infty}^{t}
  I\left(\tilde x, w; s\right)
  ds \, dw
.
  $$
  Because of (\ref{eq:I_x_zero})
  and (\ref{eq:I_x't_zero}),
  we see that $G_2\in {\Sc(N')}^{K'}$.
  Let us show that $\Fc_{x,t}G_2$  (identified with a function on $\Nc'$) 
 satisfies  Property~(R).
We have for $\hat t\not=0$ and $\hat x_3\not=0$:
$$
\Fc_{x,t} G_2(\hat x,\hat t)
=
{(\hat x_3 \hat t)}^{-1}
  \Fc_{x,t} I(\hat x; \hat t ) 
\quad\mbox{and}\quad
Q_{M-1}^{(\Fc_{x,t} G_2)}(\hat x;\hat t)
=
{(\hat x_3 \hat t)}^{-1}
Q_M^{  (\Fc_{x,t} I)}(\hat x;\hat t)
  .
  $$
By Proposition~\ref{prop_im_radon}(ii),
as $I=\radon H$,
$\Fc_{x,t}I$  (identified with a function on $\Nc'$) 
 satisfies  Property~(R),
that is
  $\radial\left(Q_M^{(\Fc_{x,t}I)}\right)$ extends to a smooth $K$-invariant function on $\Nc$ 
  which is a homogeneous polynomial in $y$ of degree $M$,
  with Schwartz coefficients in $x$.
By G.~Schwarz's Theorem,
there exists a function $\tilde Q_M\in C^\infty(\Rb^3)$ 
of the form:
$$
\tilde Q_M(r_1,r_2,r_3)=\sum_{2j_1+j_2=M} c_j(r_1) r_2^{j_1} r_3^{j_2}
\quad,\quad
c_j\in\Sc(\Rb)
,
$$
satisfying:
$$
\radial\left(Q_M^{(\Fc_{x,t}I)}\right)
=\tilde Q_M\circ \rho
.
$$
That is:
$$
Q_M^{(\Fc_{x,t}I)}(\hat x, \hat t)=
\tilde Q_M(\nn{\hat x}^2, \hat t^2,\hat x_3 \hat t)
=\sum_{2j_1+j_2=M} c_j(\nn{\hat x}^2) \hat t^{2j_1} {(\hat x_3 \hat t)}^{j_2}
.
$$
Because of (\ref{eq:I_x't_zero}),
we have:
$$
  \forall \tilde{\hat x}\in\Rb^2
    \quad,\quad
    \forall \hat t\in\Rb
    \qquad
  \Fc_{x,t}I(\tilde{\hat x},0; \hat t)
=
0
,
$$
thus the terme $c_j(\nn{\hat x}^2)$ with $j=(j_1,0)$ is zero:
we can factor out one $(\hat x_3 \hat t)$.
This implies that for $M>0$,
$\radial(Q_{M-1}^{(\Fc_{x,t} G_2)}(\hat x;\hat t))$
extends to a smooth function on $\Nc$ 
  which is a homogeneous polynomial in $y$ of degree $(M-1)$,
  with smooth coefficients in $x$.
Thus $\Fc_{x,t}G_2$ satisfies  Property~(R).
By Proposition~\ref{prop_im_radon}(ii),
  there exists $F_2\in {\Sc(N)}^K$ such that $\radon F_2=G_2$.
 As $D' G_2=I=\radon H$ and $\radon$ being 1-1 on $K$-invariant functions,
we obtain  $D F_2 =H$.
This proves~(\ref{eq_exists_F2}).
\end{proof}

Applying recursively 
Proposition~\ref{prop_Glem},
we obtain $F$ 
as a sum of functions of the form 
$\left(\Gc \left[\Delta^{j_1}D^{j_2}  F_j\right]\right)(L,0,0)$
with a rest.
As the degrees of  homogeneity of the operators $D$ and $\Delta$
with respect to the variable $y$  
are three and four respectively,
we will be interested in a sum over $2j_1+j_2 \leq M$:
\begin{cor}
  \label{cor_Glem} 
  Let $F\in {\Sc(N)}^K$.
  There exists a family $(F_j)_{j\in\Nb^2}$ of Schwartz functions
  $F_j\in {\Sc(N)}^K$ satisfying for any $M\in\Nb$:
  $$
  F- \sum_{2j_1+j_2\leq M}
  \left(\Gc \left[\Delta^{j_1}D^{j_2}  F_j\right]\right)
  (L,0,0)
  =\sum_{2j_1+j_2=M+1}\Delta^{j_1}D^{j_2} F_j
.
  $$
\end{cor}

\subsection{End of the proof}
\label{subsec_proof_thmgoal}

Here we complete the proof of Theorem~\ref{thm_goal1}.
Let $F\in {\Sc(N)}^K$,
$G=\radon F$, 
$f=\Gc F$,
$g=\Gc' G$,
$F_j$, 
the associated functions by Corollary~\ref{cor_Glem}
(consequence of the Geller-type Lemma),
and
$f_j=\Gc F_j$, $j\in \Nb^2$.
By Lemma~\ref{lem_heis_case}
we choose 
$\tilde g$, $\tilde g_j\in \Sc(\Rb^4)$ 
Schwartz extensions of $g$ and $\Gc' (\radon F_j)$, $j\in \Nb^2$, respectively.
We set
$\tilde g_{\rho'} =\tilde g\circ \rho'$.

Let us fix $M$.
For $\xi=(\tilde \xi,\xi_3)\in \Rb^3$,
setting
$r=\xi_3$
and $\lambda_l=\nn{\tilde \xi}^2/(2l+1)$,
we have $\rho'(\xi,\lambda_l)\in \Sigma_{\Dc'}$
and:
\begin{eqnarray*}
  \tilde g_{\rho'}(\xi,\lambda_l)
  &=&
  \tilde g\circ \rho'(\xi,\lambda_l)
  =
  g(\lambda_l (2l+1)+ r^2, \lambda_l^2, \lambda_l r, r^2)
  \\
  &=&
  f(\lambda_l (2l+1)+ r^2, \lambda_l^2, \lambda_l r)
  \\
  &=&
  \sum_{2j_1+j_2 \leq M}
  \lambda_l^{2j_1} {(\lambda_l r)}^{j_2}
  f_j(\lambda_l (2l+1)+ r^2, 0,0)
  \\
  &&\quad
  +
  \sum_{2j_1+j_2 = M+1}
  \lambda_l^{2j_1} {(\lambda_l r)}^{j_2}
  f_j(\lambda_l (2l+1)+ r^2, \lambda_l^2,\lambda_l r)
.
\end{eqnarray*}
Thus:
\begin{eqnarray*}
  &&\nn{  \tilde g_{\rho'}(\xi,\lambda_l)
    -\sum_{2j_1+j_2 \leq M}
    \lambda_l^{2j_1} {(\lambda_l r)}^{j_2}
    f_j(\lambda_l (2l+1)+ r^2, 0,0)}\\
  &&\qquad
  \leq 
\left(\sum_{2j_1+j_2 = M+1} \nd{\tilde g_j}_{M+1,\Rb^4}\right)
  \nn{\lambda_l}^{M+1}
.
\end{eqnarray*}
This characterises the Taylor expansion of $\tilde g_{\rho'}(\xi,.)$:
for $\xi=(\tilde \xi,\xi_3)$ first with $\tilde \xi\not=0$, and then for all
$\tilde \xi$,
we have:
\begin{eqnarray*}
  Q_M^{(\tilde g_{\rho'})}(\xi,t)
  &=&
  \sum_{2j_1+j_2= M}
  t^{2j_1} {(t \xi_3)}^{j_2}
  f_j(\nn{\xi}^2, 0,0)
  \\
  &=&
  \sum_{2j_1+j_2= M}
  {(\rho'_2(\xi,t))}^{j_1} {(\rho'_3(\xi,t))}^{j_2}
  f_j(\rho'_1(\xi,t), 0,0)
. 
\end{eqnarray*}
This shows that $\tilde g_{\rho'}$ satisfies Property~(R).
By Proposition~\ref{prop_im_restrict},
there exists $f_1\in {\Sc(\Nc)}^K$ such that $\restrict f_1= \tilde g_{\rho'}$
and $f_1=\radial g_{\rho'}$.
By G.~Schwarz's Theorem
(see also \cite[Theorem 6.1]{astengo_diblasio_ricci_08}),
there exists $\tilde f\in \Sc(\Rb^3)$
such that $f_1=\tilde f\circ \rho$.
We have:
$$
\restrict f=\tilde g_{\rho'}= \tilde f\circ \rho_{|\Nc'}=g\circ \rho' 
.
$$
For any point
$s=(\lambda(2l+1)+r^2,\lambda^2,\lambda r)
\in \Sigma_\rho$,
the point
$s'=(s,r^2)\in \Sigma_{\rho'}$ 
is in $\im \rho'$;
it follows that $s$ is in $\im \rho$ and 
$\tilde f(s)=\tilde g(s')=g(s')=f(s)$.
Thus  $\tilde f$ is an extension of $f$.

Now that we have shown that the Gelfand transform of a function
$F\in{\Sc(N)}^K$
admits a Schwartz extension,
we still have to prove the continuity of 
$F\in{\Sc(N)}^K\mapsto \Gc F\in \Sc(\Sigma_\rho)$.
We will use the following two lemmas.
The first one states 
the improvement due to Mather  \cite{mather}
of G.~Schwarz's Theorem
as well as some straightforward consequences:

\begin{lem}
\label{lem_mather}
Let $(\rho_1,\ldots,\rho_q)$ be a minimal and homogeneous Hilbert
basis for~$(\Rb^p,K)$,
and $\rho$ the corresponding Hilbert mapping.

The induced application
$\rho^*:\tilde h\mapsto \tilde h\circ \rho$ 
on $\Sc(\Rb^q)$
is split-surjective,
i.e. it admits  a linear continuous right inverse 
$\sigma:{\Sc(\Rb^p)}^K\rightarrow \Sc(\Rb^q)$
for $\rho^*$,
that is $\rho^*\circ\sigma$ is the identity mapping of ${\Sc(\Rb^p)}^K$.

We fix such $\sigma$.
For any $h\in \Sc(\im\rho)$, 
the function $h\circ \rho$ is well defined and in ${\Sc(\Rb^p)}^K$,
the function $\tilde h=\sigma (h\circ \rho) \in \Sc(\Rb^q)$
defines a Schwartz extension 
which we will call the Mather extension of $h\in \Sc(\im\rho)$.
We have:
$\tilde h\circ \rho = h\circ \rho$.
The  linear mapping 
$\tilde h\longmapsto \tilde h$
of $\Sc(\Rb^q)$
is continuous.  
\end{lem}

It is easy to check that 
the Hilbert mapping, $\rho$, of $(\Nc,K)$
is minimal and homogeneous.

The second lemma follows 
from $\restrict$ being a 1-1 continuous mapping,
from  the Closed Graph Theorem
and Lemma~\ref{lem_mather}:
\begin{lem}
To any $\tilde g\in \Sc(\Rb^4)$
such that there exists $\tilde f\in\Sc(\Rb^3)$
satisfying $\tilde f\circ \rho_{|\Nc'}=\tilde g\circ \rho'$,
we associate the Mather extension $\tilde f_1$ of $\tilde f$.
The mapping $\tilde g\mapsto \tilde f_1$ is
well-defined, continuous and linear:
\begin{equation}
  \label{eq_tildefM_tildeg}
\forall a\in\Nb
\quad\exists C>0
\quad\exists a'\in \Nb 
\quad
\nd{\tilde f_1}_{a,\Rb^3}
\leq
C \nd{\tilde g}_{a',\Rb^4}
\end{equation}
\end{lem}

Let $a_o\in \Nb$.
Let $a_1$ corresponding to $a'$ in~(\ref{eq_tildefM_tildeg})
for $a=a_o$.

Let $F\in{\Sc(N)}^K$, $G=\radon F$, $f=\Gc F$, $g=\Gc' G$.
By~(\ref{lem_heis_case_cont.}),
there exists $a_2\in\Nb$ 
such that we have indepently of~$G$:
$$
\nd{\tilde g}_{a_1,\Rb^4}
\leq 
C \nd{G}_{a_2,N}
.
$$
As $\radon$ is continuous,
there exists $a_3\in\Nb$ such that we have indepently of~$F$:
$$
\nd{\radon F}_{a_2,N}
\leq 
C \nd{F}_{a_3,N}
.
$$
Thus we have:
$$
\nd{\tilde f_1}_{a_o,\Rb^3}
\leq
C_1 \nd{\tilde g_o}_{a_1,\Rb^4}
\leq
C_2 \nd{G}_{a_2,N'}
\leq
C_3 \nd{F}_{a_3,N}
.
$$
Notice that $a_3$ and $C_3$ depend only on $a_o$,
and that $\tilde f_1$ depends 
on $F$ and also on $a_o$
because $\tilde g_o$ depends on $a_1$.

\appendix 

\section{ }

We adopt again the notation of
Section~\ref{sec_results}
and assume that $(K\ltimes N,K)$ is a Gelfand pair.
Here we give the proofs 
of Propositions~\ref{prop_ess_selfadj}, \ref{prop_joint_spectrum} and \ref{prop_gen_hula}.

\subsection{Proof of Proposition~\ref{prop_ess_selfadj}} 
\label{subsec_proof_selfadj}

This proof is an easy  generalisation of 
\cite[Lemma 5.3]{astengo_diblasio_ricci_08}
which is a similar result given in the case of the Heisenberg group,
using \cite{benson_jenkins_ratcliff_1990}.

Let us check that the operators of $\Db(N)$ are symmetric.
In fact, 
$C_c^\infty(N)$ is equipped with the Hilbert inner product
$\langle F_1,F_2 \rangle=\int_N F_1(n) \bar F_2(n) dn$.
For any $D=D_P\in\Db(N)$, $P\in\Pc(\Nc)$,
we have $\langle D F_1,F_2 \rangle =\langle F_1, D F_2 \rangle$
because:
\begin{eqnarray}
\langle D F_1,F_2 \rangle
&=&
{\left[P(i^{-1}\partial_u)
\int_N  F_1(n\exp (\sum_{j=1}^p u_j E_j)) \bar F_2(n) dn\right]}_{|u=0}
\nonumber \\
&=&
{\left[P(i^{-1}\partial_u)
\int_N  F_1(n_1)) \bar F_2(n_1\exp (-\sum_{j=1}^p u_j E_j)
dn\right]}_{|u=0}
\label{eq1_subsec_proof_selfadj}
\end{eqnarray}
after the change of variable
$n_1=n\exp (\sum_{j=1}^p u_j E_j)$.

Let us recall some facts about Gelfand pairs of the form 
$(K\ltimes N,K)$
\cite{benson_jenkins_ratcliff_1990}.
Let $\hat N$ be the set of (the classes of) unitary representations on~$N$.
For each $\pi\in \hat N$, 
let $K_\pi$ be the stabilizer of $\pi$ in~$K$. 
There exists a decomposition 
of the Hilbert space $\Hc_\pi$
into finite-dimensional irreducible subspaces $\Hc_{\pi,\alpha}$
under the projective action of $K_\pi$ on $\Hc_\pi$.
Each bounded $K$-spherical function $\phi$ on $N$  is in 1-1 correspondence
with $\pi$ and $\alpha$,
in the sense that $\phi=\phi_{\pi,\alpha}$ can be written as:
$$
\phi(n)=\int_K\langle \pi(kn) u, u\rangle dk
,
$$
where $u$ is any unit vector in $\Hc_{\pi,\alpha}$
(and $dk$ the Haar probability measure of $K$). 
Let $D\in {\Db(N)}^K$. 
For each $\pi\in \hat N$,
each subspace $\Hc_{\pi,\alpha}$
is an eigenspace for the operator $d\pi(D)$ 
and its eigenvalue is $ \mu_{\pi,\alpha,D}$ satisfies:
$ D\phi_{\pi,\alpha}=\mu_{\pi,\alpha,D}\phi_{\pi,\alpha}$.

Note that
the trace $\tr_{\Hc_{\pi}}$ of operators on $\Hc_\pi$
can be computed
as the sum over $\alpha$ of traces $\tr_{\Hc_{\pi,\alpha}}$ of operators on $\Hc_{\pi,\alpha}$.

We denote by  $\beta$ the Plancherel measure on $\hat N$:
$$
\nd{F}_2^2=
\int_{\hat N} \tr_{\Hc_\pi} \left[\pi(F){\pi(F)}^*\right] d\beta(\pi)
\quad,\quad
F\in C_c^\infty(N)
.
$$

Now let us prove Proposition~\ref{prop_ess_selfadj}.

Let $D\in {\Db(N)}^K$. 
It is easy to see that 
there exists a unique self-adjoint extension of~$D$,
whose  domain is the space of function $F\in L^2(N)$
satisfying:
$$
\int_{\hat N} 
\sum_\alpha 
\nn{\mu_{\pi,\alpha,D}}^2 
\tr_{\Hc_{\pi,\alpha}} \left[ \pi(F){\pi(F)}^*\right] 
d\beta(\pi)
<\infty
.
$$
Let us also denote by $D$ the self-adjoint extension.
Following \cite[Lemma~5.3]{astengo_diblasio_ricci_08},
we construct a realisation $E=E_D$ of the spectral resolution of $D$
in the following way.
Given $\omega$ a Borel  subset of $\Rb$,
we define the operator $E(\omega)$ on $L^2(N)$ by:
$$
\pi(E(\omega)F)
=\sum_\alpha \chi_\omega(\mu_{\pi,\alpha,D})
\pi(F)\Pi_{\pi,\alpha}
,
$$
where $\chi_\omega$ is the characteristic function of $\omega$ 
and $\Pi_{\pi,\alpha}$ the orthogonal projection of $\Hc_\pi$ onto
$\Hc_{\pi,\alpha}$.
Then $E=\{E(\omega)\}$ defines a resolution of the identity,
and for $F\in \Sc(N)$,
$$
\int_\Rb \xi dE(\xi)F= D F
.
$$
Therefore $E=E_D$ is the spectral resolution of $D$.

One readily checks that if $D_1,D_2\in {\Db(N)}^K$,
then for any Borel sets $\omega_1$, $\omega_2$,
the operators $E_{D_1}(\omega_1)$ and $E_{D_2}(\omega_2)$ 
commute.

\subsection{Proof of Proposition~\ref{prop_joint_spectrum}}

Let us recall the definition of the joint spectrum of a given strongly commuting family of
self-adjoint operators $T_1,\ldots, T_q$ (densily defined) on a Hilbert space~$\Hc$:
it is the set $S_{T_1,\ldots,T_q}$ of the $q$-tuples $\mu=(\mu_1,\ldots,\mu_q)\in \Rb^q$
for which there do not exist bounded operators $U_1,\ldots,U_q$ on $\Hc$ 
satisfying:
$$
\sum_{j=1}^q (\mu_j-T_j)U_j=
\sum_{j=1}^q U_j(\mu_j-T_j)=
\id_\Hc
.
$$

Let $(\rho_1,\ldots,\rho_q)$ be an ordered Hilbert basis for $(\Nc,K)$
and $\Dc_\rho$ the associated family of strongly commuting
self-adjoint operators on $L^2(N)$.

For each $\pi\in\hat N$,
we decompose its Hilbert space
$\Hc_\pi=\oplus^\perp_\alpha \Hc_{\pi,\alpha}$ 
as in the proof of Proposition~\ref{prop_ess_selfadj} 
in Section~\ref{subsec_proof_selfadj}
and we have for $j=1,\ldots,q$:
$$
{d\pi(D_{\rho_j})}_{|\Hc_{\pi,\alpha}}=\mu_j(\phi_{\pi,\alpha}) \id_{\Hc_{\pi,\alpha}}
.
$$
This implies the inclusion $\Sigma_{\Dc_\rho}\subset S_{\Dc_\rho}$.

For the converse inclusion, 
we will need the following Lemma,
an easy consequence of the Plancherel formula:
\begin{lem}
\label{lem_cq_plancherel}
If a function $m$ is  continuous and compactly-supported on the Gelfand spectrum,
then there exists a $K$-invariant function $M\in {L^2(N)}^K$ 
whose Gelfand transform is $m$. 
Furthermore, 
the convolution operator with kernel $M$ defined on~$C_c^\infty(N)$ 
extends to a bounded operator on $L^2(N)$ with operator norm~$\sup \nn{m}$.
\end{lem}
We will also use a dyadic decomposition on $\Rb^+$:
there exists a smooth, non-negative function $\psi$, 
supported in the interval $[\frac 12, 2]$ and satisfying:
$$
\forall x>0
\qquad
\sum_{a\in \Zb} \psi(2^{-a}x)=1
.
$$
We set $\psi_a(x)=\psi(2^{-a}x)$ if $a\geq1$,
and $\psi_o(x)=\sum_{a<1}\psi(2^{-a}x)$.

Let $\mu^o=(\mu_1^o,\ldots,\mu_q^o)\in \Rb^q\backslash\Sigma_{\Dc_\rho}$.
We define the fonctions $m_{a,j}$, $a\geq 0$, $j=1,\ldots,q$ by:
$$
m_{a,j}(\mu)
=
\frac{ \mu_j^o-\mu_j}{ \sum_{j'=1}^q {\mu_{j'}^o}^2 -  \mu_{j'}^2} 
\psi_a (\nn{\mu})
\quad,\quad
\mu\in\Sigma_{\Dc_\rho}
.
$$
Each function $m_{a,j}$ is continuous and compactly supported in
$\Sigma_{\Dc_\rho}$ 
and because $\mu^o$ is not in the closed set $\Sigma_{\Dc\rho}$,
there exists a constant $C=C(\mu^o)>0$, 
independent of $a$ and $j$, 
such that:
$$
\sup_{\mu\in \Sigma_{\Dc_\rho}}\nn{m_{a,j}(\mu)} \leq C 2^{-a}
.
$$
We denote by $U_{a,j}$ the convolution operator 
whose kernel admits $m_{a,j}$ as Gelfand transform;
by Lemma~\ref{lem_cq_plancherel},
this operator is bounded on~$L^2(N)$ with norm less than $C 2^{-a}$.
The operator $\sum_{a\geq 0} U_{a,j}$ is thus also a bounded operator
on $L^2(N)$, which we denote by $U_j$.
We check that for any representation $\pi\in\hat N$, 
we have on each subspace $\Hc_{\pi,\alpha}$:
$$
{\pi(U_j)}_{|\Hc_{\pi,\alpha}}
=
\frac{ \mu_j^o-\mu_j(\phi_{\pi,\alpha})}{ \sum_{j'=1}^q {\mu_{j'}^o}^2
  -  {\mu_{j'}(\phi_{\pi,\alpha})}^2} 
,
$$
from which we deduce:
$$
\sum_{j=1}^q
\pi(U_j)\left(\mu_j^o-\pi(D_j)\right)
=
\sum_{j=1}^q
\left(\mu_j^o-\pi(D_j)\right)\pi(U_j)
=\id_\Hc
.
$$
This implies:
$$
\sum_{j=1}^q
U_j\left(\mu_j^o-D_j\right)
=
\sum_{j=1}^q
\left(\mu_j^o-D_j\right)U_j
=
\id_{L^2(N)}
,
$$
that is, $\mu^o$ is not in the joint spectrum $S_{D_\rho}$.

This shows the inclusion $\Sigma_{\Dc_\rho}\supset S_{\Dc_\rho}$
and concludes the proof of Proposition~\ref{prop_joint_spectrum}.

\subsection{Proof of Proposition~\ref{prop_gen_hula}} 
\label{subsec_proof_prop_gen_hula}

With the Plancherel formula
(see proof of Proposition~\ref{prop_ess_selfadj}
in Subsection~\ref{subsec_proof_selfadj}),
it is easy to see that 
if $m\in \Sc(\Rb^d)$ and if $m(\Dc_\rho)$ is a convolution operator
whose kernel is $M\in {\Sc(N)}^K$,
then the Gelfand transform of $M$ coincides with $m$
on $\Sigma_{\Dc_\rho}$.

The proof of the rest of Proposition~\ref{prop_gen_hula}
relies mainly on the generalisation
 \cite[Theorem 5.2]{astengo_diblasio_ricci_08}
of Hulanicki's Schwartz Kernel Theorem.

We will also use the following Lemma
which is well-known to specialists, 
but for which we were not able to find a written proof:
\begin{lem}
  \label{lem_folland_stein}
Let $N$ be a graded Lie group, 
$\Nc=\Vc_1\oplus \Vc_2\oplus \ldots \Vc_l$
its graded Lie algebra,
${(X_i)}$ a basis of $\Vc_1$,
$L=-\sum X_i^2$ the associated sub-Laplacian.

For any homogeneous left-invariant differential operator $D$ on $N$ 
of degree $2d$,
there exists a constant $C=C(D)>0$
such that we have:
\begin{equation}
  \label{eq_lem_folland_stein}
\forall F\in C_c^\infty (N)
\quad
\nd{ DF}_2\leq  C \nd{L^d F}_2
.  
\end{equation}
Furthermore 
 $\tilde D= 2CL^d -D$ 
is a positive Rockland operator on~$N$.
\end{lem}
\begin{proof}
We refer to \cite[ch.6.A]{folland_stein_bk} 
for the definition and the properties 
of kernels of type $\alpha\in [0,Q[$,
where $Q$ is the homogeneous dimension of the group.
We will also use the fact that 
the sub-Laplacian, $L$, has a fundamental solution,
- $L$ being a homogeneous positive Rockland operator of order two -
and that the same is true for~$L^d$, $d=1,2,\ldots$.
Let us denote $G_d$ a fundamental solution of $L^d$.
For $2d<Q$, $G_d\in C^\infty(N \backslash \{0\})$ is homogeneous of
degree $2d-Q$
\cite{folland_1975}.

For any composition of left-invariant vector fields $X^I=X_{i_1}X_{i_2}\cdots X_{i_k}$ with $k< 2d$,
it is easy to check that
$X^I G_d\in C^\infty(N \backslash \{0\})$ 
is a homogeneous function $H$ of degree $-Q+1$, 
smooth away from the origin. 
One further differentiation gives a homogeneous distribution of degree~$-Q$. 
Being a derivative, 
it automatically satisfies the cancellation condition~(63) of~\cite[ch.XIII.5.3]{stein_bible}. 
In fact, let $\phi$ be a function supported on the unit ball and
normalized in the $C^1$-norm. 
For any $X\in\Vc_1$ and $r>0$,
$$
\langle XH,\phi(r\cdot)\rangle=-r\int_N H(x)X\phi(rx)\,dx=-\int_N
H(x)X\phi(x)\,dx
,
$$
which is bounded independently of $\phi$ and $r$.

This implies that for every $I$ of length $2d$, the kernel $X^I G_d$ 
satisfies the $L^2$-boundedness condition (6.3) of
\cite[ch.6.A]{folland_stein_bk},
and thus is of type~0.
The operator $X^I L^{-d}$ being $L^2(N)$-bounded,
we have:
\begin{equation}
  \label{eq_lem_folland_stein_XI}
\forall F\in C_c^\infty (N)
\quad
\nd{X^I F}_2\leq C\nd{L^d F}_2
.  
 \end{equation}

If $2d\ge Q$, 
$L^d$ does not have a homogeneous fundamental solution, 
but, according to \cite{geller_1983}, 
it has a fundamental solution $G_d$ 
which is the sum of two terms, 
one homogeneous of degree $2d-Q$, 
and the other of the form $P(x)\log|x|$, 
where $P$ is a polynomial, 
homogeneous of degree $2d-Q$, 
and $|x|$ is any smooth homogeneous norm on $N$. 
This implies that, 
if the length $k$ of $I$ satisfies $2d-Q<k<2d$, 
then $X^IG_d$ is a homogeneous function of degree $-Q+2d-k$. 
We can then repeat the previous argument to conclude that (\ref{eq_lem_folland_stein_XI}) holds for every $d$.

Let $D$ be a homogeneous left-invariant differential operator on~$N$ 
of degree $2d$.
As $D$ can be written as a linear combination of monomials $X^I$, with $I$ of degree $2d$,
we see that the property
(\ref{eq_lem_folland_stein_XI}) implies (\ref{eq_lem_folland_stein}).
Let $C=C(D)$ be the $L^2$-operator norm of $DL^{-d}$. 
In particular the $L^2(N)$-norm of the operator $D {(C L^d)}^{-1}$ is one
and $I-\frac 12 D {(C L^d)}^{-1}$ is an invertible operator on $L^2(N)$.
The differential operator $\tilde D= 2C L^d -D$
is a $2d$-homogeneous,
left-invariant, 
symmetric and positive on $C_c^\infty(N)$.
To finish the proof,
it remains to prove the defining property of Rockland operators,
that is, 
for any non-trivial, irreducible, unitary representation~$\pi$ of~$N$,
$\pi(\tilde D)$ is injective on smooth vectors;
this is true because we can write:
$$  \pi(\tilde D)=\pi(2C L^d -D)
=2C\pi\left(I-\frac 12 D {(C L^d)}^{-1}\right)
{\pi( L)}^d
,
$$
and $I-\frac 12 D {(C L^d)}^{-1}$ is invertible
and $L$ a Rockland operator.
\end{proof}

Before proving Proposition~\ref{prop_gen_hula},
let us define some notation.
We equip the two-step nilpotent Lie algebra  $\Nc$ 
with an Euclidean product 
such that $K$ acts orthogonally.
$K$ stabilises the centre $\Zc$ of $\Nc$, 
and its orthogonal complement $\Vc=\Zc^\perp$.
The decomposition $\Nc=\Vc\oplus \Zc$ endows $N$ with a structure of graded Lie group.
$Q=\dim \Vc+2\dim \Zc$ is the homogeneous dimension of the group.
For the symmetrisation mapping,
we assume that the basis ${(E_i)}_{i=1}^p$ is given as a basis
${(E_i)}_{i=1}^{p_1}$ of $\Vc$
completed with a basis ${(E_i)}_{i=p_1}^p$ of $\Zc$.
As the action of $K$ on $\Pc(\Nc)$
respects the degree-graduation in both the $\Zc$ and $\Vc$-variables,
there exist bi-homogeneous Hilbert basis $\{\rho_1,\ldots,\rho_q\}$
in the sense that
each polynomial $\rho_j$ is 
homogeneous in the $\Zc$-variables and in the $\Vc$-variables.
For a bi-homogeneous Hilbert basis $\{\rho_1,\ldots,\rho_q\}$,
we denote by $d_j^{(1)}$ the degree of homogenity of $\rho_j$ in the $\Vc$-variables,
and by $d_j^{(2)}$ the degree of homogenity of $\rho_j$ in the
$\Zc$-variables;
$d_j= d_j^{(1)}+2d_j^{(2)}$ is the degree of homogeneity of the operator $D_{\rho_j}$
for the structure of graded Lie group of $N$.

Let us start the proof of Proposition~\ref{prop_gen_hula}.
We notice that it suffices to show the result for one Hilbert mapping
because of the existence of a polynomial mapping between two Hilbert mappings.
We choose a bi-homogeneous ordered Hilbert basis
$\rho=(\rho_1,\ldots,\rho_q)$ 
with the two following properties. 
First
$\rho_1(\sum_{j=1}^p u_j E_j)= \sum_{j=1}^{p_1} \nn{u_j}^2$.
Second,
the polynomials $\rho_1,\ldots,\rho_{q_1}$ are of even degree of
homogeneity in the $\Vc$-variables
and 
the polynomials $\rho_{q_1+1},\ldots,\rho_q$ are of odd degree of
homogeneity in the $\Vc$-variables.

Let  $m$ be in~$\Sc(\Rb^q)$.
$S$ denotes 
the set of all the sequences 
$\epsilon : \{q_1+1,\ldots,q\}\rightarrow \{0,1\}$.
Using Whitney's Theorem or G.~Schwarz's Theorem,
there exists a family of Schwartz functions 
${(\tilde m_\epsilon)}_{\epsilon\in  S}$,
$ \tilde m_\epsilon\in\Sc(\Rb^q)$  
satisfying for all $(r_1,\ldots,r_q)\in \Rb^q$:
$$
m(r_1,\ldots,r_q)
=
\sum_{\epsilon \in S}
 r^{\epsilon}    \tilde  m_\epsilon(r_1,\ldots,r_{q_1}, r_{q_1+1}^2,\ldots,r_q^2)
,
$$
where we use the notation
$r^\epsilon=r_{q_1+1}^{\epsilon(q_1+1)}\ldots r_q ^{\epsilon(q)}$.

The operator $\tilde D_1 = D_{\rho_1}$ is the sub-Laplacian of~$N$
which is a positive Rockland operator.
By Lemma~\ref{lem_folland_stein}
there exist constants $c_j$, $j=2,\ldots,q$ such that 
\begin{itemize}
\item 
for $j=2,\ldots,q_1$, 
the operator 
$\tilde D_j=-D_{\rho_j}+ c_j D_{\rho_1}^{\frac {d_j} 2}$ 
is a positive Rockland operator on~$N$
\item 
for $j=q_1+1,\ldots,q$, 
the operator
$\tilde D_j=-D_{\rho_j}^2+ c_j D_{\rho_1}^{d_j}$ 
is a positive Rockland operator on~$N$
\end{itemize}
For $r=(r_1,\ldots,r_q)\in \Rb^q$, 
we set ${[A(r)]}_1=r_1$ and:
$$
\begin{array}{rcl@{\quad,\quad}l}
{[A(r)]}_j&=& -r_j+c_jr_1^{d_j/2} 
& j=2,\ldots,q_1\\
{[A(r)]}_j&=& -r_j+c_jr_1^{d_j} 
& j=q_1+1,\ldots,q
\end{array}
.
$$
This defines an application $A:\Rb^q\rightarrow \Rb^q$ 
which is a $C^\infty$-diffeomorphism of $\Rb^q$
and whose Jacobian equals ${(-1)^{q-1}}$ at any point.
Thus if $h$ is in $\Sc(\Rb^q)$ then $h\circ A^{-1}$ is in $\Sc(\Rb^q)$. 

We have:
\begin{equation}
  \label{eq_m_Drho}
m(D_{\rho_1},\ldots,D_{\rho_q})
=
\sum_{\epsilon \in S}
 D_{\rho'}^\epsilon
\tilde m_\epsilon(D_{\rho_1},\ldots,D_{\rho_{q_1}},D_{\rho_{q_1+1}}^2,\ldots,D_{\rho_q}^2)
 \end{equation}
(using the notation
$ D_{\rho'}^\epsilon=D_{\rho_{q_1+1}}^{\epsilon(q_1+1)}\ldots D_{\rho_q }^{\epsilon(q)}$)
and:
\begin{equation}
  \label{eq_mji_Drho}
\tilde m_\epsilon(D_{\rho_1},\ldots,D_{\rho_{q_1}}, D_{\rho_{q_1+1}}^2,\ldots,D_{\rho_q}^2)
=
\tilde m_\epsilon\circ A^{-1} (\tilde D_1,\ldots,\tilde D_q)
\end{equation}
Each operator given by~(\ref{eq_mji_Drho})
is a Schwartz multiplier $\tilde m_\epsilon\circ A^{-1} \in \Sc(\Rb^q)$
of a strongly commutative family of positive Rockland operators $\tilde D_j$, $j=1,\ldots q$.
By \cite[Theorem 5.2]{astengo_diblasio_ricci_08},
it is a convolution operator with a Schwartz kernel 
$M_{\tilde  m_\epsilon\circ A^{-1}, (\tilde D_j)}$.
Because of the expression~(\ref{eq_m_Drho}),
we deduce that 
the operator $m(D_{\rho_1},\ldots,D_{\rho_q})$
is also a convolution operator with a Schwartz kernel
$M_{m,\Dc_\rho}=\sum_{\epsilon \in S} D_{\rho'}^\epsilon M_{\tilde
  m_\epsilon\circ A^{-1}, (\tilde D_j)}$.

The continuity of 
$m\in \Sc(\Rb^q)\mapsto M_{m,\Dc_\rho}\in {\Sc(N)}^K$
is a direct consequence of the following facts:
\begin{itemize}
\item 
by Schwarz-Mather's Theorem,
the mappings 
$m \in \Sc(\Rb^q)\mapsto m_\epsilon\in\Sc(\Rb^{q_1}\times {[0,\infty[}^{q-q_1})$,
$\epsilon\in S$, are continuous
\item
the application $A$ being a $C^\infty$-diffeomorphism of $\Rb^q$
with ${(-1)}^{q-1}$ as  jacobian,
the mapping 
${A^{-1}}^*:h\in\Sc(\Rb^q)\mapsto h\circ A^{-1}\in\Sc(\Rb^q)$ 
is continuous
\item
by \cite[Theorem 5.2]{astengo_diblasio_ricci_08},
the application that maps $m\in \Sc(\Rb^q)$ 
to the kernel $M_{m,(\tilde D_j)}$ of the operator $m(\tilde D_1, \ldots, \tilde D_q)$
is continuous
\end{itemize}

The proof of Proposition~\ref{prop_gen_hula} is thus complete.

\bibliographystyle{alpha}

\bibliography{article}

\end{document}